\documentclass[10pt]{amsart}
\usepackage{amsthm,amsfonts,amsmath,amssymb,cite}

\newtheorem{theorem}{Theorem}[section]
\newtheorem{definition}{Definition}[section]
\newtheorem{example}{Example}[section]

\newtheorem{proposition}{Proposition}[section]

\newtheorem{question}{Question}[section]
\numberwithin{equation}{section}
\parskip\smallskipamount
\def\DJ{\leavevmode\setbox0=\hbox{D}\kern0pt\rlap
 {\kern.04em\raise.188\ht0\hbox{-}}D}
\def\dj{\leavevmode\setbox0=\hbox{d}\kern0pt\rlap
 {\kern.215em\raise.46\ht0\hbox{-}}d}

\def\RR{\mathbb{R}}

\def\NN{\mathbb{N}}

\numberwithin{equation}{section}
\newcommand{\Int}{\mathop{\mathrm{int}}}

\begin{document}

\title[Cone metric spaces]{A new survey: Cone metric spaces}
\author[S. Aleksi\'c, Z. Kadelburg, Z.D. Mitrovi\'c, S.Radenovi\'c]
{Suzana Aleksi\'{c}$^1$, Zoran Kadelburg$^2$, Zoran D.
Mitrovi\'{c}$^3$ and Stojan Radenovi\'{c}$^4$}
\date{}
\maketitle

\noindent$^1$ University of Kragujevac, Faculty of Sciences,
Department of Mathematics and Informatics, Radoja Domanovi\'{c}a
12, 34000 Kragujevac, Serbia
\medskip

\noindent$^2$ University of Belgrade, Faculty of Mathematics,
Studentski trg 16, 11000 Beograd, Serbia
\medskip

\noindent$^3$ University of Banja Luka, Faculty of Electrical
Engineering, Patre 5, 78000 Banja Luka, Bosnia and Herzegovina
\medskip

\noindent$^4$ University of Belgrade, Faculty of Mechanical
Engineering, Department of Mathematics, Kraljice Marije 16, 11120
Belgrade 35, Serbia
\newline
\medskip

E-mail: \texttt{suzanasimic@kg.ac.rs} (S. Aleksi\'{c}),
\texttt{kadelbur@matf.bg.ac.rs} (Z. Kadelburg),
\texttt{zmitrovic@etfbl.net} (Z. D. Mitrovi\'{c}),
\texttt{radens@beotel.rs} (S. Radenovi\'{c})

\bigskip

\noindent \textbf{Abstract.} The purpose of this new survey paper
is, among other things, to collect in one place most of the
articles on cone (abstract, K-metric) spaces, published after
2007. This list can be useful to young researchers trying to work
in this part of functional and nonlinear analysis. On the other
hand, the existing review papers on cone metric spaces are
updated.

The main contribution is the observation that it is usually
redundant to treat the case when the underlying cone is solid and
non-normal. Namely, using simple properties of cones and Minkowski
functionals, it is shown that the problems can be usually reduced
to the case when the cone is normal, even with the respective norm
being monotone. Thus, we offer a synthesis of the respective fixed
point problems arriving at the conclusion that they can be reduced
to their standard metric counterparts. However, this does not mean
that the whole theory of cone metric spaces is redundant, since
some of the problems remain which cannot be treated in this way,
which is also shown in the present article.
\medskip

 \medskip


\bigskip

\section{Introduction}\label{S1}

Since 1922, when S. Banach proved in his PhD thesis (see
\cite{SB1922}) the celebrated Contraction Principle for
self-mappings on a complete metric space, several hundreds of
researchers have tried to generalize or improve it. Basically,
these generalizations were done in two directions---either the
contractive condition was replaced by some more general one, or
the environment of metric spaces was widened. In the first
direction, a lot of improved results appeared, like Kannan's,
Chatterjea's, Zamfirescu's, Hardy-Rogers's, \'Ciri\'c's,
Meir-Keeler's, Boyd-Wong's, to mention just a few. The other
direction of investigations included introduction of semimetric,
quasimetric, symmetric, partial metric, $b$-metric, and many other
classes of spaces.

Of course, not all of these attempts were useful in applications,
which should be the main motive for such investigations. Some of
them were even not real generalizations, since the obtained
results appeared to be equivalent to the already known ones.

It seems that \DJ.\ Kurepa was the first to replace the set $\RR$
of real numbers by an arbitrary partially ordered set as a
codomain of a metric (see~\cite{Ku}). This approach was used by
several authors in mid-20th century, who used various names
(abstract metric spaces, $K$-metric spaces, \dots) for spaces thus
obtained (see, e.g., \cite{K,kz,V,Z}). The applications were
mostly in numerical analysis.

The interest in such spaces increased after the 2007 paper
\cite{HZ} of L.~G.~Huang and X.~Zhang, who re-introduced the
mentioned type of spaces under the name of \textit{cone metric
spaces}. Their approach included the use of interior points of the
cone which defined the partial order of the space. The original
paper \cite{HZ} used also the assumption of normality of the cone,
but later most of the results were obtained for non-normal cones,
however with more complicated proofs.

It was spotted very early that (topological vector) cone metric
spaces were metrizable (see, e.g., \cite{dU,F,KRR1,KP}). However,
this does not necessarily mean that all fixed point results in
cone metric spaces reduce to their standard metric counterparts,
since at least some of them still depend on the particular cone
that is used. Hence, the research in these spaces has continued
and more than 300 papers has appeared---most of them are cited in
the list of references of the present paper (we probably skipped
some of them, but it was not on purpose).

The intention of this article is, first, to recollect the basic
properties of cones in (topological) vector spaces which are
important for research in cone metric spaces. In particular
(non)-normal and (non)-solid cones are investigated in some
details. Then, it is shown that each solid cone in a topological
vector space can be essentially replaced by a solid and normal
cone in a normed space (even with normal constant equal to~$1$).
Thus, most of the theory of (TVS) cone metric spaces can be easily
reduced to the respective problems in standard metric spaces.

It has to be noted that some problems still remain to which the
previous assertion cannot be applied--- we will mention some of
them in Section \ref{Sec4}.

The present paper can also be treated as a continuation and
refinement of the existing survey articles on this theme, like
\cite{JKR,Kh2,PDP}.

\section{Topological and ordered vector spaces}\label{Sec2}

We start our paper with some basic facts about ordered topological
vector spaces. For more details, the reader may consult any book
on this topic, e.g., \cite{AlI,D,ZS,HH,WN}.

\subsection{Topological vector spaces}

Throughout the paper, $E$ will denote a vector space over the
field $\RR$ of real numbers, with zero-vector denoted by~$\theta$.
A topology $\tau$ on $E$ is called a \textit{vector topology} if
linear operations $(x,y)\mapsto x+y$ and
$(\lambda,x)\mapsto\lambda x$ are continuous on $E\times E$,
resp.\ $\RR\times E$. In this case, $(E,\tau)$ is called a
\textit{$($real\/$)$ topological vector space\/} (TVS). If a TVS
has a base of neighbourhoods of $\theta$ whose elements are
(closed and absolutely) convex sets in $E$, then it is called
\textit{locally convex\/} (LCS).

The topology of an LCS can be characterized using seminorms.
Namely, a nonnegative real function $p$ on a vector space $E$ is
called a \textit{seminorm} if: $1^{\circ}$~$p(x+y)\le p(x)+p(y)$
for $x,y\in E$; \ $2^{\circ}$~$p(\lambda x)=|\lambda|\,p(x)$ for
$x\in E$, $\lambda\in\RR$. It is easy to prove that a function
$p:E\to\RR$ is a seminorm if and only if $p$ is the
\textit{Minkowski functional\/} of an absolutely convex and
absorbing set $M\subset E$, i.e., $p=p_M$ is given by
\[
x\mapsto p_M(x)=\inf\{\,\lambda>0\mid x\in\lambda M\,\}.
\]
Here, one can take $M=\{\,x\in E\mid p(x)<1\,\}$. The following
proposition provides the mentioned characterization.

\begin{proposition}\label{prop1}
Let $\{\,p_i\mid i\in I\,\}$ be an arbitrary family of seminorms
on a vector space $E$. If\/ $U_i=\{\,x\in E\mid p_i(x)<1\,\}$,
$i\in I$, then the family of all scalar multiples $\lambda U$,
$\lambda>0$, where $U$ runs through finite intersections
$U=\bigcap_{j=1}^nU_j$, forms a base of the neighbourhoods of
$\theta$ for a locally convex topology $\tau$ on $E$ in which all
the seminorms $p_i$ are continuous. Each LCS $(E,\tau)$ can be
described in such a way. If we put in the previous construction
$U_i=\{\,x\in E\mid p_i(x)\le1\,\}$, we obtain a base of
neighbourhoods of $\theta$ constituted of closed absolutely convex
sets.
\end{proposition}

The easiest (and well known) examples of locally convex spaces are
normed and, in particular, Banach spaces. The topology is then
generated by a single norm which is the Minkowski functional of
the unit ball.

\subsection{Cones in ordered topological vector spaces}

Let $(E,\tau)$ be a real TVS. A subset $C$ of $E$ is called a
\textit{cone} if: $1^{\circ}$~$C$ is closed, nonempty and
$C\ne\{\theta\}$; \ $2^{\circ}$~$x,y\in C$ and $\lambda,\mu\ge0$
imply that $\lambda x+\mu y\in C$; \
$3^{\circ}$~$C\cap(-C)=\{\theta\}$. Obviously, each cone generates
a partial order $\preceq_C$ in $E$ by
\[
x\preceq_C y\iff y-x\in C
\]
(we will write just $\preceq$ if it is obvious which is the
respective cone). And conversely, a partial order $\preceq$ on $E$
generates a cone $C=\{\,x\in E\mid x\succeq\theta\,\}$. If
$\preceq$ is a partial order on a TVS $E$, then $(E,\tau,\preceq)$
is called an \textit{ordered topological vector space\/} (OTVS).
If $x\preceq y$ and $x\ne y$, we will write $x\prec y$.

The most important examples of OTVS's are special cases of the
following one.

\begin{example}\label{Ex1}
{\rm Let $X$ be a nonempty set and $E$ be the set of all
real-valued functions $f$ defined on~$X$. Let a relation $\preceq$
be defined on $E$ by $f\preceq g$ if and only if $f(x)\le g(x)$
for all $x\in X$. If $\tau$ is a vector topology on $E$ such that
the respective cone $C=\{\,f\in E\mid f(x)\ge0,\;x\in X\,\}$ is
$\tau$-closed, then $(E,\tau,\preceq)$ is an OTVS.}
\end{example}

If $(E,\preceq)$ is an ordered vector space (with the respective
cone~$C$) and $x,y\in E$, $x\preceq y$, then the set
$$
[x,y]=\{\,z\in E\mid x\preceq z\preceq y\,\}=(x+C)\cap(y-C)
$$
is called an {\it order-interval}. A subset $A$ of $E$ is {\it
order-bounded\/} if it is contained in some order-interval. It is
{\it order-convex\/} if $[x,y]\subset A$ whenever $x,y\in A$ and
$x\preceq y$. It has to be remarked that the families of {\it
convex\/} and {\it order-convex\/} subsets of an ordered vector
space are incomparable in general.

An OTVS $(E,\tau,\preceq)$ is said to be \textit{order-convex\/}
if it has a base of neighborhoods of $\theta$ consisting of
order-convex subsets. In this case the order cone $C$ is said to
be \textit{normal}. In the case of a normed space, this condition
means that the unit ball is order-convex, which is equivalent to
the condition that there is a number $K$ such that $x,y\in E$ and
$\theta\preceq x\preceq y$ imply that $\|x\|\leq K\|y\|$. The
smallest constant $K$ satisfying the last inequality is called the
\textit{normal constant\/} of~$C$.

Equivalently, the cone $C$ in $(E,\|\cdot\|,\preceq_C)$ is normal
if the Sandwich Theorem holds, i.e., if
\begin{equation}\label{normal}
(\forall n)\;x_n\preceq y_n\preceq z_n \text{ and
}\lim_{n\to\infty}x_n=\lim_{n\to\infty}z_n=x\text{ imply }
\lim_{n\to\infty}y_n=x.
\end{equation}
In particular, if a cone is normal with the constant equal to~$1$,
i.e., if $\theta\preceq x\preceq y$ implies that $\|x\|\leq
\|y\|$, the cone is called \textit{monotone\/} w.r.t.\ the given
norm.

The following is an example of a non-normal cone.

\begin{example}\label{JV}{\rm\cite{V}}
{\rm Let $E=C_{\mathbb{R}}^1[0,1]$ with $\|x\|=\|x\|_\infty
+\|x^{\prime}\|_\infty$ and $C=\{\,x\in E\mid x(t)\geq
0,\;t\in[0,1]\,\}$. Consider, for example, $x_n(t)=\frac{t^n}n$
and $y_n(t)=\frac 1n$, $n\in\NN$. Then $\theta \preceq x_n\preceq
y_n$, and $\lim_{n\to\infty}y_n=\theta$, but
$\|x_n\|=\max_{t\in[0,1]}|\frac{t^n}n|+\max_{t\in[0,1]}|t^{n-1}|
=\frac 1n+1>1$; hence $\{x_n\}$ does not converge to zero. It
follows by \eqref{normal} that $C$ is a non-normal cone.}
\end{example}

The cone $C$ in an OTVS $(E,\tau)$ is called \textit{solid\/} if
it has a nonempty interior $\Int C$.

There are numerous examples of cones which are solid and normal.
Perhaps the easiest is the following one.

\begin{example}\label{Ex3}
{\rm  Let $E=\RR^n$ with the Euclidean topology $\tau$. Let the
cone $C$ be given as in Example \ref{Ex1}, i.e.,
$C=\{\,(x_i)_{i=1}^n\mid x_i\ge0,\;i=1,2,\dots,n\,\}$. Then it is
easy to see that $C$ is a solid and normal cone in~$(E,\tau)$.}
\end{example}

Example \ref{JV} shows that there exists a cone which is solid and
non-normal. A lot of spaces which are important in functional
analysis have cones which are normal and non-solid. We present
just two of them.

\begin{example}\label{Ex4}
{\rm (see, e.g., \cite{J}) Consider $E=c_0$ (the standard Banach
space of all real $0$-sequences) and let $C=\{\,x=\{x_n\}\mid
x_n\ge0,\;n\in\NN\,\}$. It follows easily that $C$ is a normal
cone in~$E$. Let us prove that it has an empty interior.

Take any $x=\{x_n\}\in C$. If $x_n=0$ for $n\ge n_0$, then
obviously $x\notin\Int C$. Otherwise, for arbitrary
$\varepsilon>0$, there exists $n_0\in\NN$ such that
$x_{n}<\frac\varepsilon2$ for all $n>n_0$. Construct a sequence
$y=\{y_n\}$ by
\[
y_n=\begin{cases} x_n,&\text{if }n\le n_0,\\
-x_n,&\text{for }n>n_0.\end{cases}
\]
Then, at least one of $y_n$ is strictly negative, hence $y\notin
C$. However, $\|x-y\|=\sup_{n>n_0}|2x_n|<\varepsilon$, therefore,
the ball $B(x,\varepsilon)$ is not a subset of $C$. Thus, $\Int
C=\emptyset$.}
\end{example}

\begin{example}\label{Ex5}
{\rm \cite{NM} Let $E=L^2(\RR,\mu)\cap C(\RR)$ be equipped with
the norm $\|\cdot\|_E=\|\cdot\|_{L^2}+\|\cdot\|_\infty$, where
$\mu$ is the Lebesgue measure. Let
\[
C=\{\, h\in E \mid h(t)\geq0,\;t\in\RR\,\}
\]
be the cone of all positive elements in $E$. We will show that $C$
has an empty interior. For any $f\in C$ and each $\delta>0$,
define $N_f(\delta)=\{\,t\in\RR\mid f(t)\geq\delta\,\}$. Then
$\mu(N_f(\delta))<\infty$, and therefore $\mu(\RR\setminus
N_f(\delta))=\infty$. Note that
\begin{align*}
\mu(\RR\setminus N_f(\delta))
&=\mu\biggl(\bigcup_{r\in\mathbb{Q}}\bigl(B(r,\delta)\cap(\RR\setminus
N_f(\delta))\bigr)\biggr)\\
&\leq \sum_{r\in\mathbb{Q}}\mu \bigl( B(r,\delta)\cap(\RR\setminus
N_f(\delta))\bigr) =\infty ,
\end{align*}
which means that $\mu(A)>0$, where $A=B(s,\delta)\cap(\RR\setminus
N_f(\delta))$, for some $s\in\mathbb{Q}$. Define
\[
g(x) =\begin{cases} f(x),& x\notin A, \\
f(x)-\delta,& x\in A.
\end{cases}
\]
Then $g$ is negative on the set $A$ of positive measure and
\begin{equation*}
\| f-g\|_{L_2}+\| f-g\| _\infty  =\biggl(\int_{A}
\delta^2\,d\mu\biggr)^{1/2} + \sup_{A}  \delta  \leq (2\delta
^3)^{1/2}+\delta.
\end{equation*}
Hence, for any $f\in C$ and arbitrary $\varepsilon>0$, we can find
$g\notin C$ with $\|f-g\|_E<\varepsilon$, i.e., there is no ball
with centre $f$ that lies inside $C$. Therefore, $C$ has an empty
interior.}
\end{example}

We state now the following properties of solid cones.

$\bullet$ \cite[Proposition (2.2), page 20]{WN} $e\in\Int C$ if
and only if $[-e,e] = (C-e) \cap (e-C)$ is a $\tau$-neighbourhood
of~$\theta$.

$\bullet$ If $e\in\Int C$, then the Minkowski functional
$\|\cdot\|_e=p_{[-e,e]}$ is a norm on the vector space~$E$.
Indeed, one has just to prove that $\|x\|_e=0$ implies that
$x=\theta$. Suppose, to the contrary, that $x\ne\theta$. Then, by
the definition of infimum, it follows that there is a sequence
$\{\lambda_n\}$ of positive reals tending to $0$, and such that
$x\in\lambda_n[-e,e]$, i.e., $-\lambda_ne\preceq
x\preceq\lambda_ne$. It follows that the sequences
$\{\lambda_ne-x\}$ and $\{x+\lambda_ne\}$ both belong to $C$.
Since they converge to $-x$ and $x$, respectively, the closedness
of the cone $C$ implies that $-x,x\in C$. Hence, $x=\theta$, which
is a contradiction. See also \cite{AlI,JK}.

Note that this conclusion was stated in \cite{Kh1} [the paragraph
before Theorem 3.1] under the additional assumption that the cone
$C$ is normal. The previous proof shows that this assumption is in
fact redundant.

$\bullet$ If $e_1,e_2\in\Int C$, then the respective norms
$\|\cdot\|_{e_1}$ and $\|\cdot\|_{e_2}$ are equivalent. In
particular, the interiors of the cone $C$ w.r.t.\ the norms
$\|\cdot\|_{e_1}$ and $\|\cdot\|_{e_2}$ coincide; in fact, they
are equal to $\Int_{\tau}C$. This follows by the characterization
of interior points of a cone in an OTVS.

$\bullet$ If $\Int_{\tau}C\neq \emptyset$, then the topology
$\tau$ is Hausdorff. Indeed, if $e\in \Int C$, then the unit ball
$[-e,e]$ is a neighbourhood of $\theta$ in the topology generated
by the norm $\|\cdot\|_e$. Since this norm topology is Hausdorff
and not stronger than $\tau$, it follows that $\tau$ is Hausdorff,
too.

We will show now that $C$ is a cone, i.e., it is closed, in the
topology generated by $\|\cdot\|_e$. Let $x_n\in C$ and $x_n\to x$
as $n\to\infty$, in the norm $\|\cdot\|_e$. This means that
$-x\preceq -(x-x_n)\preceq \frac1n e$ for $n\geq n_0$. Passing to
the limit as $n\to\infty$, we get that $-x\preceq \theta$, i.e.,
$x\succeq \theta$, and so $x\in C$.

Thus, we have proved the next assertion (see also
\cite{Ale,D,KPRR,rAD,R3,VeF}).

\begin{proposition}\label{prop2}
If $C$ is a solid cone in an OTVS $(E,\tau,\preceq_C)$, then there
exists a norm $\|\cdot\|_e$ on $E$ which is monotone and such that
the cone $C$ is normal and solid w.r.t.\ this norm. In particular,
the cone $C$ has the same set of interior points, both in topology
$\tau$ and in the norm generated by $\|\cdot\|_e$.
\end{proposition}

Further, following \cite{HZ}, denote
\begin{equation}\label{ll}
x\ll y\quad \text{if and only if}\quad y-x\in\Int C.
\end{equation}
Let $\{x_n\}$ be a sequence in $(E,\tau,\preceq_C)$ and $c\in\Int
C$. We will say that $\{x_n\}$ is a \textit{$c$-sequence} if there
exists $n_0\in\NN$ such that $x_n\ll c$ whenever $n\ge n_0$. The
connection between $c$-sequences and sequences tending to $\theta$
is the following.

$\bullet$ If $x_n\to\theta$ as $n\to\infty$, then it is a
$c$-sequence for each $c\in\Int C$. Indeed, if $c\in\Int C$, then
$[-c,c]$ is a $\tau$-neighbourhood of $\theta$ and it follows that
$\Int[-c,c]=(\Int C-c)\cap(c-\Int C)$ is also  a
$\tau$-neighbourhood of $\theta$. Hence, there exists $n_0\in\NN$
such that, for all $n\ge n_0$, $x_n\in(\Int C-c)\cap(c-\Int C)$,
i.e., $c-x_n\in\Int C$, and so $x_n\ll c$.

$\bullet$ The converse of the previous assertion is not true,
i.e., $x_n\ll c$ for some $c\in\Int C$ and all $n\ge n_0$ does not
necessarily imply that $x_n\to\theta$ (see Example \ref{JV}).
However, if the cone $C$ is normal, then these two properties of a
sequence $\{x_n\}$ are equivalent.

Note also the following properties of bounded sets.

$\bullet$ If the cone $C$ is solid, then each topologically
bounded subset of $(E,\tau,\preceq_C)$ is also order-bounded,
i.e., it is contained in a set of the form $[-c,c]$ for some
$c\in\Int C$.

$\bullet$ If the cone $C$ is normal, then each order-bounded
subset of $(E,\tau,\preceq_C)$ is topologically bounded. Hence, if
the cone is both solid and normal, these two properties of subsets
of $E$ coincide.

In other words, we have the following property (see, e.g.,
\cite{V}).

\begin{proposition}\label{prop3}
If the underlying cone of an OTVS is solid and normal, then such
TVS must be an ordered normed space.
\end{proposition}

The following old result of M.\ Krein can also be useful when
dealing with cones in normed spaces.

\begin{proposition}\label{lemma1}{\rm\cite{Kr}}
A cone $C$ in a normed space $(E,\|\cdot\|)$ is normal if and only
if there exists a norm $\|\cdot\|_{1}$ on $E$, equivalent to the
given norm $\|\cdot\|$, such that the cone $C$ is monotone w.r.t.\
$\|\cdot\|_{1}$.
\end{proposition}

\section{Cone metric spaces}

\subsection{Definition and basic properties}
As was already said in Introduction, spaces with ``metrics''
having values in ordered spaces, more general than the set $\RR$
of real numbers, appeared under various names (abstract metric
spaces, $K$-metric spaces, \dots) since the mid-20th century (see,
e.g., \cite{K,kz,Ku,V,Z}). Starting from 2007 and the paper
\cite{HZ}, some version of the following definition has been
usually used.

\begin{definition}\label{Def1}
Let $X$ be a non-empty set and $(E,\tau,\preceq)$ be an OTVS. If a
mapping $d:X\times X\to E$ satisfies the conditions
\begin{itemize}
\item[(i)] $d(x,y)\succeq\theta$ for all $x,y\in X$ and
$d(x,y)=\theta$ if and only if $x=y$; \item[(ii)] $d(x,y)=d(y,x)$
for all $x,y\in X$; \item[(iii)] $d(x,z)\preceq d(x,y)+d(y,z)$ for
all $x,y,z\in X$,
\end{itemize}
then $d$ is called a \textit{cone metric} on $X$ and $(X,d)$ is
called a \textit{cone metric space}.
\end{definition}

Convergent and Cauchy sequences can be introduced in the following
way.

\begin{definition}\label{Def2}
Let $(X,d)$ be a cone metric spaces with the cone metric having
values in $(E,\tau,\preceq)$, where the underlying cone $C$ is
solid. Let $\{x_n\}$ be a sequence in~$X$. Then we say that:
\begin{enumerate}
\item the sequence $\{x_n\}$ converges to $x\in X$ if for each
$c\in\Int C$ there exists $n_0\in\NN$ such that $d(x_n,x)\ll c$
holds for all $n\ge n_0$; \item $\{x_n\}$ is a Cauchy sequence if
for each $c\in\Int C$ there exists $n_0\in\NN$ such that
$d(x_m,x_n)\ll c$ holds for all $m,n\ge n_0$; \item the space
$(X,d)$ is complete if each Cauchy sequence $\{x_n\}$ in it
converges to some $x\in X$.
\end{enumerate}
\end{definition}

Note that, according to \cite{KPRR}, in the previous definition,
equivalent notions are obtained if one replace $\ll$ by $\prec$ or
$\preceq$.

The following crucial observation (which is a consequence of
Proposition \ref{prop2}) shows that most of the problems in cone
metric spaces can be reduced to their standard metric
counterparts.

\begin{proposition}\label{equivalence}
Let $X$ be a non-empty set and $(E,\tau,\preceq_C)$ be an OTVS
with a solid cone~$C$. If $e\in\Int C$ is arbitrary, then
$D(x,y)=\|d(x,y)\|_e$ is a $($standard\/$)$ metric on $X$.
Moreover, a sequence $\{x_n\}$ is Cauchy $($convergent\/$)$ in
$(X,d)$ if and only if it is Cauchy $($convergent\/$)$ in $(X,D)$.
\end{proposition}

We will show in Section \ref{Sec4} how this approach simplifies
most of the proofs concerning fixed points of mappings.

Note that it is clear that the cone metric with values in
$(E,\|\cdot\|_e,\preceq)$ is continuous (as a function of two
variables). Also, Sandwich Theorem holds (since the cone $C$ is
normal in the new norm).

\subsection{Completion and Cantor's intersection theorem}

Most of standard notions from the setting of metric spaces, like
accumulation points of sequences, open and closed balls, open and
closed subsets, etc., are introduced in the usual way. Also,
standard properties (e.g., that closed balls are closed subsets)
can be easily deduced, based on Proposition \ref{prop2}.

As a sample, we state the following theorem on the completion of a
cone metric space and we note that, with the mentioned approach,
its proof is much easier than in \cite{Abdel}.

\begin{theorem}
Let $(X,d)$ be a cone metric space over an OTVS
$(E,\tau,\preceq_C)$ with a solid cone. Then there exists a
complete cone metric space $(\widetilde{X},\widetilde{d})$ over an
ordered normed space $(E,\|\cdot\|,\preceq_C)$ with a solid and
normal cone, containing a dense subset
$(\widetilde{X}^{*},\widetilde{d}^{*})$, isometric with $(X,d)$.
\end{theorem}

Now, we present a proof of the Cantor's intersection theorem in
the setting of cone metric spaces.

\begin{theorem}\label{Cantor}
Let $(X,d)$ be a cone metric space over an OTVS
$(E,\tau,\preceq_C)$ with a solid cone. Then $(X,d)$ is complete
if and only if every decreasing sequence $\{B_n\}$ of non-empty
closed balls in $X$ for which the sequence of diameters tends to
$\theta$ as $n\to\infty$, has a non-empty intersection, and more
precisely, there exists a point $x\in X$ such that
$\bigcap_{n=1}^{\infty}B_n=\{x\}$.
\end{theorem}

\begin{proof}
According to Propositions \ref{prop2} and \ref{lemma1}, without
loss of generality, we can assume that $(X,d)$ is a cone metric
space over an ordered normed space $(E,\|\cdot\|,\preceq_C)$ with
the cone $C$ being solid and normal. Moreover, we can assume that
the normal constant of $C$ is equal to~$1$ (i.e., the norm is
monotone w.r.t.~$C$).

Assume that $(X,d)$ is a complete cone metric space, and let
$B_n=B[x_n,c_n]$, $n\in\mathbb{N}$, be the sequence of non-empty,
closed balls in $X$ with the properties: $B_{n+1}\subset B_n$,
$n\in\mathbb{N}$, and $c_n\rightarrow\theta$ as
$n\rightarrow\infty$. We first show that $\{x_n\}$ is a Cauchy
sequence.

If $m>n$, we have $x_m\in B_m\subset B_n$ which implies that
$d(x_m,x_n)\preceq c_n$. Since $c_n\rightarrow \theta$ as
$n\rightarrow\infty$, it follows that $d(x_m,x_n)\to\theta$ as
$n\to\infty$, and consequently $\{x_n\}$ is a Cauchy sequence.

The completeness of $(X, d)$ implies the existence of $x\in X$
such that $x_n\rightarrow x$ as $n\rightarrow\infty$. Since
$x_m\in B_n$ for all $m\geq n$, and $B_n$ is a closed subset of
$X$, it follows that $x\in B_n$ for all $n\in \mathbb{N}$, i.e.,
$x\in\bigcap_{n=1}^{\infty}B_n$.

Suppose that there exists $y\in\bigcap_{n=1}^{\infty}B_n$ such
that $y\neq x$. We thus get
\[d(x,y)\preceq d(x, x_n) + d(x_n, y)\preceq 2c_n,\]
which gives $d(x,y)\preceq\theta$ and, as a consequence, we obtain
$x = y$, a contradiction. Hence, $\bigcap_{n=1}^{\infty}B_n=
\{x\}$.

Conversely, to obtain a contradiction, suppose that $(X,d)$ is not
a complete cone metric space. Then there exists a Cauchy sequence
$\{x_n\}$  which does not converge in $X$.

We construct  a strictly increasing sequence of positive integers
$\{n_i\}$ in the following way. If $c_1\gg \theta$, then there
exists $n_1\in\mathbb{N}$ such that $m > n_1$ implies $d(x_m,
x_{n_1})\ll c_1$. Now, for $c_2 =\frac{c_1}{2}$ there exists
$n_2>n_1$ such that $d(x_m, x_{n_2})\ll c_2$ holds for all
$m>n_2$. We continue in this manner.  For $c_i=\frac{c_{i-1}}{2}$,
there exists $n_i>n_{i-1}$ such that $m>n_i$ implies $d(x_m,
x_{n_i})\ll c_i$.

Let us consider the sequence $B_i=B[x_{n_i},2c_i]$,
$i\in\mathbb{N}$, such that
$c_i=\frac{c_1}{2^{i-1}}\rightarrow\theta$  as
$i\rightarrow\infty$. We first show that  $B_{i+1}\subset B_i$ for
all $i\in\mathbb{N}$. From $y\in B_{i+1}$, it follows that
$d(y,x_{n_{i+1}})\preceq2c_{i+1}$, hence that
\[d(y,x_{n_i})\preceq d(y, x_{n_{i+1}}) +
d(x_{n_{i+1}},x_{n_i})\preceq 2c_{i+1}+c_i=2\frac{c_i}{2} + c_i
=2c_i,\] and finally that $y\in B[x_{n_i}, 2c_i]$. We thus get
$y\in B_i$, which is the desired conclusion.

Let us prove that $\bigcap_{i=1}^{\infty}B_i=\emptyset$. Indeed,
if there exists an $x\in\bigcap_{i=1}^{\infty}B_i$, then $x\in
B_i$ for all $i\in\mathbb{N}$ and thus $d(x,x_i)\preceq2c_i$,
$i\in\mathbb{N}$. For $m > n_i$, we have
\[
d(x,x_m)\preceq d(x, x_{n_i}) + d(x_{n_i},x_m)\preceq2c_i+c_i=3c_i.
\]
The normality of the cone $C$ gives $d(x,x_m)\to\theta$ as
$m\to\infty$ (since $m>n_i$). This contradicts the fact that the
Cauchy sequence $\{x_n\}$ does not converge in $X$.
\end{proof}

Comparing the previous proof with the one of \cite[Theorem
3.10]{JK}), we see that our result is more general and with a
shorter proof.

\section{Fixed point results in cone metric spaces}\label{Sec4}

\subsection{Some fixed point results}

Most of the papers from the present list of references dealt with
some (common) fixed point problems in cone metric spaces. However,
the majority of the obtained results are in fact direct
consequences of the corresponding results from the standard metric
spaces. We will show this just on certain examples, but it will be
clear that the same procedure can be done in most of the other
cases.

Thus, let $(X,d)$ be a cone metric space over an OTVS
$(E,\tau,\preceq)$ with a solid cone~$C$. Let $f,g:X\to X$ be two
self-mappings satisfying $fX\subset gX$, and assume that one of
these two subsets of $X$ is complete. Consider the following
contractive conditions:

\begin{enumerate}
\item (Banach) $d(fx,fy)\preceq \lambda d(gx,gy)$,
$\lambda\in(0,1)$;

\item (Kannan) $d(fx,fy)\preceq\lambda d(gx,fx)+\mu d(gy,fy)$,
$\lambda,\mu\ge0$, $\lambda+\mu<1$;

\item (Chatterjea) $d(fx,fy)\preceq\lambda d(gx,fy)+\mu d(gy,fx)$,
$\lambda,\mu\ge0$, $\lambda+\mu<1$;

\item (Reich) $d(fx,fy)\preceq\lambda d(gx,gy)+\mu d(gx,fx)+\nu
d(gy,fy)$, $\lambda,\mu,\nu\ge0$, $\lambda+\mu+\nu<1$;

\item (Zamfirescu) one of the conditions
\[
\left\{\begin{aligned} d(fx,fy)&\preceq \lambda
d(gx,gy),\;\;\lambda\in(0,1)\\
d(fx,fy)&\preceq\lambda d(gx,fx)+\mu d(gy,fy),\;\;
\lambda,\mu\ge0,\; \lambda+\mu<1\\
d(fx,fy)&\preceq\lambda d(gx,fy)+\mu d(gy,fx),\;\;
\lambda,\mu\ge0,\; \lambda+\mu<1; \end{aligned}\right.
\]

\item (Hardy-Rogers) $d(fx,fy)\preceq \lambda d(gx,gy)+\mu
d(gx,fx)+ \nu d(gy,fy)+\xi d(gx,fy)+\pi d(gy,fx)$,
$\lambda,\mu,\nu,\xi,\pi\ge0$, $\lambda+\mu+\nu+\xi+\pi<1$;

\item (\'Ciri\'c) $d(fx,fy)\preceq\lambda u(x,y)$, for some
$u(x,y)$ belonging to one of the following sets
\begin{align*}
&\left\{d(gx,gy),\dfrac{d(gx,fx)+d(gy,fy)}2,\dfrac{d(gx,fy)+d(gy,fx)}2\right\}\\
&\left\{d(gx,gy),d(gx,fx),d(gy,fy),\dfrac{d(gx,fy)+d(gy,fx)}2\right\}
\end{align*}
where $\lambda\in(0,1)$.
\end{enumerate}

\begin{theorem}\label{CFPT}
If one of the conditions $(1)$--$(7)$ is satisfied for all $x,y\in
X$, then the mappings $f$ and $g$ have a unique point of
coincidence. Moreover, if $f$ and $g$ are weakly compatible, then
they have a unique common fixed point in~$X$.
\end{theorem}

\begin{proof}
Let $e\in\Int C$ be arbitrary. Based on Propositions \ref{prop2},
\ref{lemma1} and \ref{equivalence}, consider the metric space
$(X,D)$, where $D(x,y)=\|d(x,y)\|_e$. Then, the self-mappings
$f,g:X\to X$ satisfy one of the following contractive conditions:
\begin{enumerate}
\item[$(1')$] $D(fx,fy)\leq \lambda D(gx,gy)$, $\lambda\in(0,1)$;

\item[$(2')$] $D(fx,fy)\leq\lambda D(gx,fx)+\mu D(gy,fy)$,
$\lambda,\mu\ge0$, $\lambda+\mu<1$;

\item[$(3')$] $D(fx,fy)\leq\lambda D(gx,fy)+\mu D(gy,fx)$,
$\lambda,\mu\ge0$, $\lambda+\mu<1$;

\item[$(4')$] $D(fx,fy)\leq\lambda D(gx,gy)+\mu D(gx,fx)+\nu
D(gy,fy)$, $\lambda,\mu,\nu\ge0$, $\lambda+\mu+\nu<1$;

\item[$(5')$] one of the conditions
\[
\left\{\begin{aligned} D(fx,fy)&\leq \lambda
D(gx,gy),\;\;\lambda\in(0,1)\\
D(fx,fy)&\leq\lambda D(gx,fx)+\mu D(gy,fy),\;\;
\lambda,\mu\ge0,\; \lambda+\mu<1\\
D(fx,fy)&\leq\lambda D(gx,fy)+\mu D(gy,fx),\;\; \lambda,\mu\ge0,\;
\lambda+\mu<1; \end{aligned}\right.
\]

\item[$(6')$] $D(fx,fy)\leq \lambda D(gx,gy)+\mu D(gx,fx)+ \nu
D(gy,fy)+\xi D(gx,fy)+\pi D(gy,fx)$,
$\lambda,\mu,\nu,\xi,\pi\ge0$, $\lambda+\mu+\nu+\xi+\pi<1$;

\item[$(7')$] one of the conditions
\begin{align*}
D(fx,fy)&\leq\lambda \left\{D(gx,gy),\dfrac{D(gx,fx)+D(gy,fy)}2,\dfrac{D(gx,fy)+D(gy,fx)}2\right\}\\
D(fx,fy)&\leq\lambda \left\{D(gx,gy),D(gx,fx),D(gy,fy),\dfrac{D(gx,fy)+D(gy,fx)}2\right\}
\end{align*}
where $\lambda\in(0,1)$.
\end{enumerate}

It is well known (see, e.g., \cite{BER} for the case $g=id_X$; the
general cases are similar) that, in each of these cases, it
follows that the mappings $f$ and $g$ have a unique point of
coincidence. The conclusion in the case of weak compatibility is
also standard (see, e.g., \cite{Jungck}).
\end{proof}

For \'Ciri\'c's quasicontraction (see also \cite{BER}), but for a
single mapping, the following cone metric version can be proved in
a similar way.

\begin{theorem}
Let $(X,d)$ be a complete cone metric space over an OTVS
$(E,\tau,\preceq)$ with a solid cone~$C$. Let $f:X\to X$ be a
self-mapping such that there exists $\lambda\in(0,1)$ and, for all
$x,y\in X$ there exists
\[
u(x,y)\in\left\{d(x,y),d(x,fx),d(y,fy),d(x,fy),d(y,fx)\right\}
\]
satisfying  $d(fx,fy)\preceq\lambda u(x,y)$. Then $f$ has a unique
fixed point in~$X$.
\end{theorem}

\subsection{Results that cannot be obtained in the mentioned way}

Certain fixed point problems in cone metric spaces cannot be
treated as in the previous subsection. This is always the case
when the given contractive condition cannot be transformed to an
appropriate condition in the corresponding metric space $(X,D)$.
As an example of this kind we recall a Boyd-Wong type result.

Let $(E,\tau,\preceq_C)$ be an OTVS with a solid cone. A mapping
$\varphi:C\to C$ is called a \textit{comparison function} if:
(i)~$\varphi$ is nondecreasing w.r.t.\ $\preceq$;
(ii)~$\varphi(\theta)=\theta$ and $\theta\prec\varphi(c)\prec c$
for $c\in C\setminus\{\theta\}$; (iii)~$c\in\Int C$ implies
$c-\varphi(c)\in\Int C$; (iv)~if $c\in C\setminus\{\theta\}$ and
$e\in\Int C$, then there exists $n_0\in\NN$ such that
$\varphi^n(c)\ll e$ for each $n\ge n_0$.

\begin{theorem}{\rm\cite[Theorem 2.1]{Ar}}
Let $(X,d)$ be a complete cone metric space over
$(E,\tau,\preceq_C)$ with a solid cone. Let the mappings $f,g:X\to
X$ be weakly compatible and suppose that for some comparison
function $\varphi:C\to C$ and for all $x,y\in X$ there exists
$u\in\{d(gx,gy),d(gx,fx),d(gy,fy)\}$, such that
$d(fx,fy)\preceq\varphi(u)$. Then $f$ and $g$ have a unique common
fixed point in~$X$.
\end{theorem}

An additional example is provided by the following result.

\begin{theorem}{\rm\cite{CvRa}}
Let $(X,d)$ be a complete cone metric space over
$(E,\|\cdot\|,\preceq_C)$ with a solid cone. Let $f:X\to X$ and
suppose that there exists a positive bounded operator on $E$ with
the spectral radius $r(A)<1$, such that $d(fx,fy)\preceq
A(d(x,y))$ for all $x,y\in X$. Then $f$ has a unique fixed point
in ~$X$.
\end{theorem}

In this case, passing to the space $(E,\|\cdot\|_e,\preceq_C)$
cannot help, since the two norms need not be equivalent and we
have no information about the boundedness of $A$ and (if it is
bounded) how $r(A)$ has changed.

A very important special class of fixed point results are the
Caristi-type ones, and the huge literature is devoted to them.
When cone metric spaces are concerned, they are treated in
details, e.g., in the papers \cite{AK,ChBN,KaaD,Kh1,Kh2,Kh,wLO1},
so we will give here only some basic information.

The following result was proved in \cite{AK}.

\begin{theorem}{\rm \cite[Theorem 2]{AK}}
Let $(X,d)$ be a complete cone metric space over an OTVP
$(E,\|\cdot\|,\preceq_C)$ with a solid cone. Let $F:X\to C$ be a
lower-semi-continuous map and let $f:X\to X$ satisfy the condition
\[
d(x,fx)\preceq F(x)-F(f(x))
\]
for any $x\in X$. Then $f$ has a fixed point.
\end{theorem}

It is not clear whether, by passing to the space
$(E,\|\cdot\|_e,\preceq)$ using the Minkowski functional, this
theorem reduces to the classical Caristi result. However, it is
proved in \cite{Kh} that this is still the case if $e\in\Int C$
can be chosen so that the corresponding Minkowski functional
$p_{[-e,e]}$ is additive. It is an open question whether this can
be done without this assumption.

\section{Some ``hybrid'' spaces}\label{Sec5}

As has been already mentioned, besides cone metric spaces, a lot
of other generalizations of standard metric spaces have been
introduced and fixed point problems in these settings have been
investigated by several authors. There are among them some
``hybrid'' spaces, i.e., spaces where axioms of several types are
used simultaneously. Such are, e.g., cone $b$-metric spaces,
cone-$G$-metric spaces, cone rectangular metric spaces, partial
cone metric spaces, cone spaces with $c$-distance and some others.
We will recall here definitions of two of these classes and how
the work in these classes can be easily reduced to the already
known ones. Similar conclusions hold for the other mentioned
classes, too.

\subsection{Cone $b$-metric spaces}

Cone $b$-metric spaces were introduced (under a different name) in
\cite{CV1} and some fixed point results were presented in this
setting.

\begin{definition}
Let $X$ be a non-empty set, $(E,\tau,\preceq)$ be an OTVS and
$s\ge1$ be a real number. If a mapping $d:X\times X\to E$
satisfies the conditions
\begin{itemize}
\item[(i)] $d(x,y)\succeq\theta$ for all $x,y\in X$ and
$d(x,y)=\theta$ if and only if $x=y$; \item[(ii)] $d(x,y)=d(y,x)$
for all $x,y\in X$; \item[(iii)] $d(x,z)\preceq s(d(x,y)+d(y,z))$
for all $x,y,z\in X$,
\end{itemize}
then $d$ is called a \textit{cone $b$-metric} on $X$ and $(X,d)$
is called a \textit{cone $b$-metric space with parameter~$s$}.
\end{definition}

Later, these and other researchers obtained a lot of (common)
fixed point results (see,
\cite{AMAR,BXShCR,FaAbRa,FeMaFi,GNRR,Maitra,OlBrOp,ShiXu,StM,XuHuang,ZhXuDoGo}.
However, similarly as in the case of cone metric spaces, most of
these results can be reduced to the case of classical $b$-metric
spaces of Bakhtin \cite{Bakh} and Czerwik \cite{SCz}. Namely,
similarly as Proposition \ref{equivalence}, the following can be
deduced from Proposition \ref{prop2}.

\begin{proposition}\label{equivalence-b}
Let $X$ be a non-empty set, $(E,\tau,\preceq_C)$ be an OTVS with a
solid cone~$C$ and $s\ge1$ be a real number. If $e\in\Int C$ is
arbitrary, then $D(x,y)=\|d(x,y)\|_e$ is a $b$-metric on $X$.
Moreover, a sequence $\{x_n\}$ is Cauchy $($convergent\/$)$ in
$(X,d)$ if and only if it is Cauchy $($convergent\/$)$ in $(X,D)$.
\end{proposition}


\subsection{Cone rectangular spaces}
The following definition was given in \cite{aZa}.

\begin{definition}
Let $X$ be a non-empty set and $(E,\tau,\preceq)$ be an OTVS. If a
mapping $d:X\times X\to E$ satisfies the conditions
\begin{itemize}
\item[(i)] $d(x,y)\succeq\theta$ for all $x,y\in X$ and
$d(x,y)=\theta$ if and only if $x=y$; \item[(ii)] $d(x,y)=d(y,x)$
for all $x,y\in X$; \item[(iii)] $d(x,y)\preceq
d(x,w)+d(w,z)+d(z,y)$ for all $x,y\in X$ and all distinct points
$w,z\in X\setminus\{x,y\}$,
\end{itemize}
then $d$ is called a \textit{cone rectangular metric} on $X$ and
$(X,d)$ is called a \textit{cone rectangular metric space}.
\end{definition}

In the papers \cite{aZa,jL,MaShSh,MaShSe2,RasSal,TcH,VeF}, some
fixed point results in cone rectangular spaces were obtained. The
following proposition shows how most of these results can be
easily reduced to the corresponding known ones in the environment
of rectangular spaces defined by Branciari in \cite{Bran}.

\begin{proposition}\label{equivalence-rect}
Let $X$ be a non-empty set and $(E,\tau,\preceq_C)$ be an OTVS
with a solid cone~$C$. If $e\in\Int C$ is arbitrary, then
$D(x,y)=\|d(x,y)\|_e$ is a rectangular metric on $X$. Moreover, a
sequence $\{x_n\}$ is Cauchy $($convergent\/$)$ in $(X,d)$ if and
only if it is Cauchy $($convergent\/$)$ in $(X,D)$.
\end{proposition}

\section{The case of a non-solid cone}

In this section, we will briefly consider the case when the
underlying cone is non-solid, but normal.

Let $(E,\|\cdot\|,\preceq_C)$ be an ordered normed space and let
the cone $C$ be normal, but non-solid (see, e.g., Examples
\ref{Ex4} and \ref{Ex5}). The definition of a cone metric $d$ on
an nonempty set $X$, over~$E$, is the same as Definition
\ref{Def1}. However, definition of convergent and Cauchy
sequences, as well as of the completeness, must be modified.

\begin{definition}
Let $(X,d)$ be a cone metric space over $(E,\|\cdot\|,\preceq_C)$,
with the cone $C$ being normal $($possibly non-solid\/$)$. A
sequence $\{x_n\}$ in $X$ is said to converge to $x\in X$ if
$d(x_n,x)\to\theta$ as $n\to\infty$ $($the convergence in the
sense of norm in~$E$\/$)$. Similarly, Cauchy sequences are
introduced, as well as the completeness of the space.
\end{definition}

In the view of Proposition \ref{lemma1}, we can assume that the
normal constant of $C$ is equal to~$1$, and further consider just
the standard metric space $(X,D)$, where $D(x,y)=\|d(x,y)\|$.

Open (resp.\ closed) balls in $(X,d)$ can be introduced as
follows: let $x\in X$ and $c\in C\setminus\{\theta\}$. The open
(resp.\ closed) ball with centre $x$ and radius $c$ are given by
$B(x,c)=\{y\in X\mid d(y,x)\prec c\}$ (resp.\ $B(x,c)=\{y\in X\mid
d(y,x)\preceq c\}$). It can be easily shown that an open (closed)
ball is an open (closed) subset of $(X,d)$. We will show this,
e.g., in the case of a closed ball.

Thus, let $y_n\in B[x,c]$ and $y_n\to y$ as $n\to\infty$ in
$(X,d)$. Then we have $d(y,x)\preceq d(y,y_n)+d(y_n,x)\preceq
d(y,y_n)+c$. By the Sandwich Theorem (which holds since the cone
$C$ is normal), we get that $d(y,x)\preceq c$, which means that
$y\in B(x,c)$.

However, note that the collection of open balls $\{\,B(x,c)\mid
x\in X,\,c\in C\setminus\{\theta\}\,\}$ does not necessarily form
a base of topology on $X$. Indeed, it is not sure that the
intersection of two open balls contains a ball. For example, for
$X=\RR^2$, $C=\{\,(x,y)\mid x,y\ge0\,\}$ and the points
$c_1=(1,0)$, $c_2=(0,1)$ from $C\setminus\{\theta\}$, there is no
$c\in C\setminus\{\theta\}$ such that $c\prec c_1$ and $c\prec
c_2$.

Nevertheless, the theorem on completion holds in this case, too.

We conclude with the following open question.

\begin{question}
Construct a cone in a real TVS which is both non-normal and
non-solid.
\end{question}

\subsection*{Acknowledgement}
The first and second author are grateful to the Ministry of
Education, Science and Technological Development of Serbia, Grants
No.\ 174024 and 174002, respectively.

\end{document}